\documentclass[a4paper,11pt]{article}
\usepackage{amsmath,amssymb,enumerate}
\usepackage{amsthm,color}

\usepackage{hyperref}
\hypersetup{
setpagesize=false,
 bookmarksnumbered=true,%
 bookmarksopen=true,%
 colorlinks=true,%
 linkcolor=blue,
 citecolor=red
}

\newcommand{\R}{\mathbb{R}}

\newcommand{\N}{\mathbb{N}}
\newcommand{\ep}{\varepsilon}
\newcommand{\pa}{\partial}

\DeclareMathOperator{\supp}{supp}

\newtheorem{theorem}{Theorem}[section]
\newtheorem{lemma}[theorem]{Lemma}

\theoremstyle{remark}
\newtheorem{remark}{Remark}[section]
\theoremstyle{definition}

\newtheorem{definition}{Definition}[section]

\numberwithin{equation}{section}

\makeatletter
\def\@cite#1#2{[{{\bfseries #1}\if@tempswa , #2\fi}]}
\makeatother                  
\begin{document}
\begin{center}
\Large{{\bf
Finite time blowup of solutions to semilinear wave 
\\
equation in an exterior domain
}}
\end{center}

\vspace{5pt}

\begin{center}
Motohiro Sobajima%
\footnote{
Department of Mathematics, 
Faculty of Science and Technology, Tokyo University of Science,  
2641 Yamazaki, Noda-shi, Chiba, 278-8510, Japan,  
E-mail:\ {\tt msobajima1984@gmail.com}}
\footnote{Partially supported 
by Grant-in-Aid for Young Scientists Research 
No.18K13445. }
and
Kyouhei Wakasa%
\footnote{
Department of Mathematics, 
Faculty of Science and Technology, Tokyo University of Science,  
2641 Yamazaki, Noda-shi, Chiba, 278-8510, Japan,  
E-mail:\ {\tt wakasa\_kyouhei@ma.noda.tus.ac.jp}}
\end{center}

\newenvironment{summary}{\vspace{.5\baselineskip}\begin{list}{}{%
     \setlength{\baselineskip}{0.85\baselineskip}
     \setlength{\topsep}{0pt}
     \setlength{\leftmargin}{12mm}
     \setlength{\rightmargin}{12mm}
     \setlength{\listparindent}{0mm}
     \setlength{\itemindent}{\listparindent}
     \setlength{\parsep}{0pt}
     \item\relax}}{\end{list}\vspace{.5\baselineskip}}
\begin{summary}
{\footnotesize {\bf Abstract.}
We consider the initial-boundary value problem 
of semilinear wave equation with nonlinearity $|u|^p$ 
in exterior domain in $\R^N$ $(N\geq 3)$. 
Especially, the lifespan of 
blowup solutions with small initial data 
are studied. 
The result gives upper bounds of lifespan 
which is essentially the same as the Cauchy problem in $\R^N$.
At least in the case $N=4$, their estimates are sharp in view of 
the work by Zha--Zhou \cite{ZZ15}. 
The idea of the proof is to use special solutions 
to linear wave equation with Dirichlet boundary condition
which are constructed via an argument based on Wakasa--Yordanov \cite{YWpre}.
}
\end{summary}

{\footnotesize{\it Mathematics Subject Classification}\/ (2010): %
Primary: 
35L05, 35L20, 35B44.
}

{\footnotesize{\it Key words and phrases}\/: %
Semilinear wave equations, 
blowup, upper bound of lifespan,
exterior domain
}

\section{Introduction}\label{intro}
In this paper we consider the semilinear wave equations 
in an exterior domain in $\R^N$ $(N\geq 3)$:
\begin{equation}
\label{NW}
\begin{cases}
\pa_t^2u(x,t) -\Delta u(x,t)=|u(x,t)|^{p}, 
&(x,t)\in \Omega\times (0,T),
\\
u(x,t)=0
&(x,t)\in \pa\Omega\times (0,T),
\\
u (x,0)=\ep f(x)
&x\in \Omega, 
\\
\pa_tu (x,0)=\ep g(x)
&x\in \Omega,
\end{cases}
\end{equation}
where $\pa_t=\pa/\pa t$, 
$\Delta=\sum_{j=1}^N\pa^2/\pa x_j^2$, $1<p\leq \frac{N}{N-2}$, $T>0$
and $\Omega\subset \R^N$ satisfies that 
$\R^N\setminus \Omega$ is bounded and $\pa\Omega$ is a smooth boundary. 
The pair $(f,g)$ is given (the shape of) initial data 
and the parameter $\ep>0$ describes the size (smallness) of the initial data. 
The interest of the present paper is to study the profile of solutions to
\eqref{NW} with sufficiently small initial data.
Here the pair $(f,g)$ satisfies
\begin{align}\label{ass.g}
(f,g)\in C_0^\infty(\Omega), \quad \supp (f,g)\subset \overline{B(0,r_0)},
\end{align}
where $B(0,r)=\{x\in \R^N\;;|x|<r\}$.

This kind of study of global existence and blowup of solutions to 
\eqref{NW} has been discussed 
since the pioneering work of John \cite{John79} when 
$\Omega=\R^N$ with $N=3$:
\begin{equation}
\label{NWcauchy}
\begin{cases}
\pa_t^2u(x,t) -\Delta u(x,t)=|u(x,t)|^{p}, 
&(x,t)\in \R^N\times (0,T),
\\
u (x,0)=\ep f(x)
&x\in \R^N, 
\\
\pa_tu (x,0)=\ep g(x)
&x\in \R^N. 
\end{cases}
\end{equation}
It is shown in \cite{John79}  that 
the critical exponent of \eqref{NWcauchy} is determined as $p=1+\sqrt{2}$, that is, 
\begin{itemize}
\item if $1<p<1+\sqrt{2}$, then 
the solution of \eqref{NWcauchy} blows up in finite time 
for ``positive'' initial data;
\item if $p>1+\sqrt{2}$, then 
there exists a global solution with small initial data. 
\end{itemize}
After that, Strauss \cite{Strauss81} conjectured that 
the critical exponent of \eqref{NWcauchy} for general dimension $N$ is 
given by  
\begin{align*}
p_S(N)=\sup\{p>1\;;\;\gamma(N,p)>0\}, \quad \gamma(N,p)=2+(N+1)p-(N-1)p^2.
\end{align*}
Now $p_S(N)$ is called the Strauss exponent. 
Including blowup phenomena in the critical situation, 
this conjecture was solved until the works Yordanov--Zhang \cite{YZ06} and Zhou \cite{Zhou07}. 
The further study for blowup solutions can be found in the literature. 
Especially, the behavior of lifespan (maximal existence time) of 
blowup solutions with small initial data is intensively discussed (see e.g., 
Lindblad \cite{Lindblad90}, Takamura--Wakasa \cite{TW11}, 
Zhou--Han \cite{ZH14}, Ikeda--Sobajima--Wakasa \cite{IkSoWa_pre} and the references therein).
Here the definition of lifespan is given as follows:
\begin{equation*}
T_\ep:=T(\ep f,\ep g)
=\sup\{T>0\;;\;\text{there exists a
unique weak solution $u$ of \eqref{NW} in $(0,T)$}\}.
\end{equation*}
The precise behavior of 
lifespan of small solutions is given by  
\[
T_\ep\approx
\begin{cases}
C\ep^{-\frac{2p(p-1)}{\gamma(N,p)}}
&
{\rm if\ }1<p<p_S(N),
\\
\exp(C\ep^{-p(p-1)})
&
{\rm if\ }p=p_S(N)
\end{cases}
\]
when $\ep>0$ is sufficiently small.

Of course, there are many investigations dealing with 
the exterior problem \eqref{NW} of semilinear wave equations. 
The significant difference to the initial value problem 
is the effect of reflection at the boundary 
and the lack of symmetry such as scale-invariance, 
rotation-invariance and so on. 
For the existence of global-in-time solutions 
to \eqref{NW} has been discussed in 
Du--Metcalfe--Sogge--Zhou \cite{DMSZ08} 
and
Hidano--Metcalfe--Smith--Sogge--Zhou \cite{HMSSZ10} 
when $p_S(N)<p<\frac{N+3}{N-1}$ and $N=3,4$. 
The sharp lower bounds for lifespan of solutions are shown in Yu \cite{Yu11} 
$2<p<p_S(3)$ with $N=3$; 
Zhou--Han \cite{ZH11} proved sharp upper bounds in the case  
$1<p<p_S(N)$ and $N\geq 3$. 
For the critical case $p=p_S(N)$, 
Zha--Zhou \cite{ZZ15} discussed 
the case $N=4$ and $p=p_S(4)=2$
and proved the lower bound 
$T_\ep \geq \exp(C\ep^{-2})$ 
which seems sharp from the lifespan 
estimate for the Cauchy problem.
The upper bounds for the critical cases 
are shown in 
Lai--Zhou \cite{LZ15} for $N=3$ 
and 
Lai--Zhou \cite{LZ16} for $N\geq 5$. 
We should point out that 
in the two dimensional case 
there are some blowup results 
for small initial data 
(see Li--Wang \cite{LW12} and Lai--Zhou \cite{LZ18}),
however, 
precise estimates for lifespan 
are not treated so far.


As far as the author's knowledge, 
sharp upper bound of lifespan for two and four dimensional
cases are unknown. Moreover, 
the proofs of the blowup in previous works 
(including studies of \eqref{NWcauchy})
depend on the positivity of initial data, 
especially in the higher dimensional case $N\geq 4$. 
In contrast, 
such a restriction in the whole space case is recently removed 
by using a framework of test function methods 
in Ikeda--Sobajima--Wakasa \cite{IkSoWa_pre}.

The purpose of the present paper is to 
prove blowup of solutions to \eqref{NW} 
with sharp upper bound of lifespan when $\Omega_0=\R^N\setminus \overline{B(0,1)}$
without positivity assumption in the pointwise sense as in \cite{IkSoWa_pre}.

The following is the main result of the present paper. 
\begin{theorem}\label{main}
Let $N\geq 3$, 
$\Omega=\Omega_0(=\R^N\setminus \overline{B(0,1)})$ 
and 
 $U(x)=1-|x|^{2-N}$.
Let the pair $(f,g)$ satisfy \eqref{ass.g} with
\begin{equation}\label{eq:ass.pos}
\int_{\Omega_0}g(x)U(x)\,dx>0.
\end{equation} 
If $1<p\leq p_S(N)$, then $T_\ep<\infty$ for every $\ep>0$ with the following upper bounds:
there exists a constant $\ep_0>0$ (independent of $\ep$) such that 
for every $\ep\in (0,\ep_0]$, 
\[
T_\ep\leq 
\begin{cases}
C\ep^{-\frac{2p(p-1)}{\gamma(N,p)}}
&
{\rm if\ }1<p<p_S(N),
\\
\exp(C\ep^{-p(p-1)})
&
{\rm if\ }p=p_S(N).
\end{cases}
\]
\end{theorem}
In the proof of Theorem \ref{main}, 
the positive harmonic function $U$ 
satisfying boundary condition
\begin{align*}
\begin{cases}
\Delta U=0
 & \text{in}\ \Omega_0
\\
U=0
 & \text{on}\ \pa\Omega_0
\\
U>0
 & \text{in}\ \Omega_0
\end{cases}
\end{align*} 
plays an important role as in the previous works concerning 
upper bounds for lifespan for exterior problem. 
However, the strategy of the proof 
in the present paper is  quite different from 
those. 
Our technique is based on 
the test function method for wave equations 
developed in Ikeda--Sobajima--Wakasa \cite{IkSoWa_pre}. 
This argument requires the special 
solutions of corresponding linear wave equation 
having slowly decaying property. 
To construct this kind of solution, 
we used the construction by Wakasa--Yordanov \cite{YWpre}.

\begin{remark}
Comparing the previous results for 
upper bounds for lifespan,  
we do not assume positivity of 
initial data in the pointwise sense. 
Moreover, our assumption \eqref{eq:ass.pos} means 
the quantity $\int_\Omega \pa_t uU\,dx$
is always positive
\[
\frac{d}{dt}
\int_\Omega \pa_t u U\,dx
=
\int_\Omega (\Delta u + |u|^p)U\,dx\geq 0.
\]
This may be meaningful, 
in fact, in the whole space case 
the condition $\int_{\R^N}g\,dx\neq 0$ 
is sometimes imposed to see the precise behavior 
of lifespan with respect to $\ep\ll1$. 
This poses that $\int_{\R^N}gU_{\R^N}\,dx$ 
with the positive harmonic function 
$U_{\R^N}=1$ is crucial for the lifespan estimates. 
\end{remark}

\begin{remark}
In Lai--Zhou \cite{LZ16}, to find the lifespan estimates 
they essentially assumed that the support of initial data 
is far away from boundary.
Theorem \ref{main} 
allows us to consider the initial data which 
have the support close to the boundary.
\end{remark}

\begin{remark}
Our technique is also applicable to 
the problem with nonlinearity $|\pa_t u|^p$
and the one of their weakly coupled system.
\end{remark}

The present paper 
is organized as follows. 
In Section \ref{prelim}, 
we construct 
special solutions of corresponding linear wave equation 
having slowly decaying property 
by separation of variables and the argument in Wakasa--Yordanov \cite{YWpre}. 
We also give their fundamental profiles in Section \ref{prelim}. 
Section \ref{mainproof}
is devoted to prove Theorem \ref{main} 
by using test function method based on 
the one in Ikeda--Sobajima--Wakasa \cite{IkSoWa_pre}.

\section{Preliminaries}\label{prelim}
First we consider a class of special solutions to the 
linear wave equation 
with Dirichlet boundary condition
\begin{equation}\label{eq:lin_wave}
\begin{cases}
\pa_t^2\Phi(x,t) -\Delta \Phi(x,t)=0,
&(x,t)\in \Omega_0\times (0,T),
\\
\Phi(x,t)=0,
&(x,t)\in \pa\Omega_0\times (0,T).
\end{cases}
\end{equation}
The aim of this section is to construct 
a positive solution of the linear wave equation in 
the space-time domain
\begin{equation}\label{Q}
\mathcal{Q}_1=\{(x,t)\in \Omega_0\times (0,t)\;;\;
|x|<t\}
\end{equation}
having polynomial decay of arbitrary order.

\subsection{Solutions of wave equation by separation of variables}

To begin with, 
we consider solutions by separation of variables of the form
\[
\Phi(x,t)=e^{-\lambda t}\varphi_\lambda(x), \quad (x,t)\in \Omega_0\times (0,\infty)
\] 
which has an exponential decay.
Then by \eqref{eq:lin_wave} this is equivalent to 
the following elliptic equation 
related to the eigenvalue problem of 
Laplace operator with Dirichlet boundary condition:
\begin{equation}\label{eq:elliptic}
\begin{cases}
\lambda^2 \varphi_\lambda(x)-\Delta \varphi_\lambda(x)=0,
&x\in \Omega_0,
\\
\varphi_\lambda(x)=0,
&x\in\pa\Omega_0,
\\
\varphi_\lambda(x)>0,
&x\in\Omega_0,
\end{cases}
\end{equation}
where $\lambda\geq 0$ is a parameter. Here we will construct 
a family $\{\varphi_\lambda\}_{\lambda>0}$ 
which is continuous with respect to $\lambda$ in a suitable sense. 
If $\lambda=0$, then $\varphi_0$ is nothing but a positive 
harmonic function on $\Omega_0$ satisfying the Dirichlet boundary condition, 
and therefore, we first fix 
\[
\varphi_0(x)=U(x)=1-|x|^{2-N}, \quad x\in \Omega_0.
\]
Then by using modified Bessel functions $I_\nu$ and $K_\nu$, 
we define the family of functions $\{\varphi_\lambda\}_{\lambda>0}$ as follows:
\begin{definition}
For $N\geq 3$ and $\lambda>0$, define 
\[
\varphi_\lambda(x):=
\psi_1(\lambda |x|)
-
\frac{I_\nu(\lambda)}{K_\nu(\lambda)}
\psi_2(\lambda |x|), 
\quad x\in \Omega_0
\]
where $\nu=\frac{N-2}{2}>0$ and  
\[
\psi_1(z)=2^{\nu}\Gamma(\nu+1)z^{-\nu}I_\nu(z),
\quad
\psi_2(z)=2^{\nu}\Gamma(\nu+1)z^{\nu}K_\nu(z)
\]
(for the detailed information about modified Bessel functions, see e.g., Beals--Wong \cite{BW}). 
\end{definition}
\begin{remark}
The function $\psi_1(|x|)$ can be represented by 
\[
\psi_1(|x|)=
\frac{1}{|S^{N-1}|}
\int_{S^{N-1}}e^{x\cdot \omega}\,d\omega
\]
which has been introduced in Yordanov--Zhang \cite{YZ06} 
and used many times in the previous papers listed in Section \ref{intro}. 
\end{remark}

To analyse the behavior of $\varphi_\lambda$, 
we use the precise behavior of $I_\nu$ and $K_\nu$ listed in the following lemma.
\begin{lemma}
Let $\mu>0$. Then $I_\mu$ and $K_\mu$ are smooth positive functions 
satisfying 
\begin{equation}\label{eq:mbe.ori}
z^2y''(z)+zy'(z)=\left(z^2+\mu^2\right)y(z), \quad z>0
\end{equation}
with the following properties
\begin{gather}
\label{eq:near0}
\lim_{z\downarrow 0}
\Big(z^{-\mu}I_\mu(z)\Big)=\frac{1}{2^\mu\Gamma(\mu+1)}, 
\quad
\lim_{z\downarrow 0}
\Big(z^{\mu}K_\nu(z)\Big)=2^{\mu-1}\Gamma(\mu),
\\
\label{eq:nearinf}
\lim_{z\to \infty}\Big(\frac{z^{\frac{1}{2}}}{e^z}I_\mu(z)\Big)=\frac{1}{\sqrt{2\pi}}, 
\quad
\lim_{z\to \infty}\Big(\frac{z^{\frac{1}{2}}}{e^z}K_\mu(z)\Big)=\sqrt{\frac{\pi}{2}},
\\
\label{eq:derivative}
\frac{d}{dz}\Big(z^{-\nu}I_\mu(z)\Big)=z^{-\mu}I_{\mu+1}(z),
\quad
\frac{d}{dz}\Big(z^{\nu}K_\mu(z)\Big)=-z^{-\mu}K_{\mu+1}(z).
\end{gather}
\end{lemma}
Then we have
\begin{lemma}\label{lem:phi_lambda}
The family $\{\varphi_\lambda\}_{\lambda>0}$ has 
the following properties.
\begin{itemize}
\item[\bf (i)] for every $\lambda>0$, $\varphi_\lambda$ satisfies \eqref{eq:elliptic}.
\item[\bf (ii)] the map $(x,\lambda)\in \Omega_0\times(0,\infty)\mapsto \varphi_\lambda(x)$ is continuous.
\item[\bf (iii)] for every $x\in \Omega_0$, one has
\[
\lim_{\lambda\downarrow 0}\varphi_\lambda(x)=U(x).
\]
\item[\bf (iv)] 
for every $\lambda>0$, 
\[
\varphi_\lambda(x)\geq U(x)\psi_1(\lambda|x|), \quad x\in \Omega_0.
\]
\item[\bf (v)] 
there exists a constant $C_\nu>0$ such that 
for every $\lambda\in (0,1]$, 
\[
\varphi_\lambda(x)\leq C_\nu U(x)\psi_1(\lambda|x|), \quad x\in \Omega_0.
\]
\end{itemize}
\end{lemma}

\begin{proof}
The assertion {\bf (ii)} is obvious by the construction of $\psi_1$ and 
$\psi_2$. 
{\bf (i)} is also  verified 
because the pair $(\psi_1,\psi_2)$ is the fundamental 
system of the following ordinary differential equation:
\begin{equation*}
\psi''(r)+\frac{N-1}{r}\psi'(r)=\psi(r), \quad r>0
\end{equation*}
which is equivalent (via the change of functions 
$v(z)=z^{\frac{N-2}{2}}\psi(z)$) to the modified Bessel equation \eqref{eq:mbe.ori}
with the parameter $\mu=\nu=\frac{N-2}{2}$.
By using \eqref{eq:derivative} with $\varphi_\lambda\equiv 0$ on $\pa\Omega_0$, 
$\varphi_\lambda$ satisfies \eqref{eq:elliptic}.
For {\bf (iii)}, noting that 
\begin{align*}
\varphi_\lambda(x)=
2^\nu\Gamma(\nu+1)(\lambda r)^{-\nu}I_\nu(\lambda r)
\left(
1-
\frac{\lambda^{-\nu}I_\nu(\lambda)}{(\lambda r)^{-\nu}I_\nu(\lambda r)}
\cdot
\frac{(\lambda r)^{\nu}K_{\nu}(\lambda r)}{\lambda^{\nu}K_{\nu}(\lambda)}
\cdot r^{2-N}
\right)
\end{align*}
with the notation $r=|x|$, we have $\lim_{\lambda\downarrow 0}\varphi_\lambda(x)=1-r^{2-N}=U(x)$.
For {\bf (iv)}, we define 
\[
\widetilde{\varphi}_{\lambda}(x):=
U(x)\psi_1(\lambda|x|)-\varphi_\lambda(x), 
\quad x\in \overline{\Omega}_0
\]
with arbitrary fixed $\lambda>0$. 
Note that by (i) and \eqref{eq:derivative}, we have 
\begin{align*}
\Delta\widetilde{\varphi}_{\lambda}(x)
&=
(\Delta U)\psi_1(\lambda|x|)
+
2\nabla U\cdot\nabla\big(\psi_1(\lambda|x|)\big)
+
U\Delta (\psi_1(\lambda|x|))
-
\Delta \varphi_{\lambda}(x)
\\
&=
2\nabla U\cdot\nabla\big(\psi_1(\lambda|x|)\big)
+
\lambda^2U\psi_1(\lambda|x|)
-
\lambda^2\varphi_{\lambda}(x)
\\
&\geq
\lambda^2\widetilde{\varphi}_{\lambda}(x).
\end{align*}
Moreover, by \eqref{eq:nearinf} we see that for sufficiently large $R_\lambda$, 
\begin{align*}
\widetilde{\varphi}_{\lambda}(x)
&= (1-U(x))\psi_1(\lambda|x|)+\frac{I_\nu(\lambda)}{K_\nu(\lambda)}\psi_2(\lambda|x|)
\\
&=
-\lambda^{2\nu}\psi_1(\lambda|x|)
\left(
(\lambda |x|)^{-2\nu} 
-\frac{I_\nu(\lambda)}{K_\nu(\lambda)}
\frac{K_\nu(\lambda|x|)}{I_\nu(\lambda|x|)}
\right)\leq 0,
 \quad x\in \R^N\setminus B(0,R_\lambda).
\end{align*}
Therefore
we have
\begin{equation*}
\begin{cases}
\lambda^2 \widetilde{\varphi}_{\lambda}(x)
-\Delta\widetilde{\varphi}_{\lambda}(x)
\leq 0,
&x\in \Omega_0,
\\
\widetilde{\varphi}_{\lambda}(x)=0,
&x\in\pa\Omega_0,
\\
\widetilde{\varphi}_{\lambda}(x)\leq 0,
&\text{for }\ x\in\R^N\setminus B(0,R_\lambda).
\end{cases}
\end{equation*}
The maximum principle implies $\widetilde{\varphi}_{\lambda}\leq 0$ on $\Omega_0$
and hence 
$\varphi_{\lambda}(x)\geq U(x)\psi_1(\lambda|x|)$ $(x\in\Omega_0)$ is verified.
Finally we prove {\bf (v)}. We put 
\begin{gather*}
c_{1,\nu}=\inf_{z\in (0,2)}\Big(z^{-\nu}I_\nu(z)\Big)
\leq 
\sup_{z\in (0,2)}\Big(z^{-\nu}I_\nu(z)\Big)=C_{1,\nu}, 
\\
c_{2,\nu}=\inf_{z\in (0,2)}\Big(z^{\nu}K_\nu(z)\Big)
\leq 
\sup_{z\in (0,2)}\Big(z^{\nu}K_\nu(z)\Big)=C_{2,\nu}, 
\end{gather*}
which are all finite by \eqref{eq:near0}. 
By \eqref{eq:derivative} we see that for every $x\in \Omega_0$ and $\lambda>0
$,
\begin{align*}
\frac{\pa\varphi_\lambda}{\pa r}(x)
&
=
\left(
\lambda \psi_1'(\lambda r)
-
\frac{I_\nu(\lambda)}{K_\nu(\lambda)}
\lambda \psi_2'(\lambda r)
\right)
\\
&=
2^{\nu}\Gamma(\nu+1)
\left(
\lambda (\lambda r)^{-\nu}I_{\nu+1}(\lambda r)
+
\frac{I_\nu(\lambda)}{K_\nu(\lambda)}
\lambda (\lambda r)^{-\nu}K_{\nu+1}(\lambda r)
\right).
\end{align*}
If $|x|\leq 2$ and $\lambda\in (0,1]$, then 
\begin{align*}
\frac{\pa\varphi_\lambda}{\pa r}(x)
&\leq 
2^{\nu}\Gamma(\nu+1)
\left(
C_{1,\nu+1}\lambda^2r
+
\frac{C_{1,\nu}C_{2,\nu+1}}{c_{2,\nu}}
r^{-2\nu-1}
\right)
\\
&\leq 
2^{\nu}\Gamma(\nu+1)
\left(
4^{\nu+1}C_{1,\nu+1}
+
\frac{C_{1,\nu}C_{2,\nu+1}}{c_{2,\nu}}
\right)r^{-2\nu-1}.
\end{align*}
This with Dirichlet boundary condition yields that 
for $x\in \Omega_0\cap \overline{B(0,2)}$,
\[
\varphi_\lambda(x)\leq 
2^{\nu-1}\Gamma(\nu)
\left(
4^{\nu+1}C_{1,\nu+1}
+
\frac{C_{1,\nu}C_{2,\nu+1}}{c_{2,\nu}}
\right)U(x).
\]
If $|x|\geq 2$, by the definition of $\varphi_\lambda$ 
and the monotonicity of $U(x)$ (with respect to $r=|x|$), 
we see 
\[
\varphi_\lambda(x)
\leq 
\psi_1(\lambda|x|)
=
[U(x)]^{-1}U(x)\psi_1(\lambda|x|)
\leq 
(1-2^{2-N})^{-1}U(x)\psi_1(\lambda|x|)
\]
We obtain the desired upper bound for $\varphi_\lambda$.
\end{proof}
\subsection{Slowly decaying solutions of wave equation}
Next we construct a family of solutions having 
polynomial decay of arbitrary order. 

Before the construction of solutions to the problem 
with Dirichlet boundary condition \eqref{eq:lin_wave}, 
we consider the following formula 
describing the connection between the modified Bessel function $I_{\nu}(z)$
and the Gauss hypergeometric function $F(\cdot,\cdot,\cdot;z)$
in the ``light cone''
\[
\mathcal{Q}_0=\{(x,t)\in\R^N\times(0,\infty)\;;\;|x|<t\}.
\]
\begin{lemma}\label{lem:finite}
Let $\beta>0$. If $(x,t)\in \mathcal{Q}_0$, then
\[
\frac{1}{\Gamma(\beta)}
\int_0^{\infty} e^{-\lambda t}\psi_1(\lambda |x|)\lambda^{\beta-1}\,d\lambda
=
t^{-\beta}
F\left(\frac{\beta}{2},\frac{\beta+1}{2},\frac{N}{2};\frac{|x|^2}{t^2}\right),
\]
where $F(a,b,c;z)$ is the Gauss hypergeometric function defined as
\[
F(a,b,c;z)=\sum_{n=0}^\infty
\frac{(a)_n(b)_n}{(c)_n}\frac{z^n}{n!}, \quad |z|<1
\]
with the Pochhammer symbol $(d)_0=1$ and $(d)_n=\prod_{k=1}^{n}(d+k-1)$ for $n\in\N$.
\end{lemma}
\begin{proof}
By the asymptotic profile of $I_\nu$, 
we have 
$\psi_1(x)\leq e^{|x|}$ 
and then 
\begin{align*}
e^{-\lambda t}\psi_1(\lambda |x|)\lambda^{\beta-1}
&\leq 
e^{-\lambda(t-|x|)}\lambda^{\beta-1}.
\end{align*}
This implies that if $|x|<t$, then the function 
\[
v(x,t)=
\int_0^{\infty} e^{-\lambda t}\psi_1(\lambda |x|)\lambda^{\beta-1}\,d\lambda
\]
is well-defined. By similar argument, we also have $v\in C^2(\mathcal{Q}_0)$.
Observe that for every $\lambda>0$, 
$v_\lambda=e^{-\lambda t}\psi_1(\lambda|x|)$
satisfies $\pa_t^2v_\lambda-\Delta v_\lambda=0$ in $\mathcal{Q}_0$.
Therefore $v$ also satisfies $\pa_t^2v-\Delta v=0$ on $\mathcal{Q}_0$. 
Moreover, we see from the change of variables $\mu=\lambda s$ that 
\begin{align*}
s^{\beta}v(sx,st)
&=
s^{\beta}\int_0^{\infty} e^{-\lambda st}\psi_1(\lambda |sx|)\lambda^{\beta-1}\,d\lambda
\\
&=
\int_0^{\infty} e^{-\mu t}\psi_1(\mu|x|)\mu^{\beta-1}\,d\mu
\\
&=
v(x,t).
\end{align*}
Noting that $v(0,t)=\Gamma(\beta)t^{-\beta}$, 
by \cite[Lemma 2.1]{IkedaSobajima2} (with $\mu'=0$ in their notation) we obtain the desired equality.
\end{proof}
Now we introduce the family of solutions to \eqref{eq:lin_wave}, 
which plays a crucial role in the present paper. 
\begin{definition}
For $\beta>0$, 
\[
\Phi_\beta(x,t)=
\frac{1}{\Gamma(\beta)}\int_0^{1} e^{-\lambda t}\varphi_\lambda(x)\lambda^{\beta-1}\,d\lambda, \quad (x,t)\in \mathcal{Q},
\]
where $\mathcal{Q}$ is as in \eqref{Q}.
Note that $\Phi_\beta$
is well-defined by virtue of Lemma \ref{lem:phi_lambda} {\bf (v)} and Lemma \ref{lem:finite}.
\end{definition}

The following lemma is mainly used in the proof of main result in this paper.

\begin{lemma}\label{lem:Phi_beta}
The functions $\{\Phi_\beta\}_{\beta> 0}$ satisfy the following 
properties:
\begin{itemize}
\item[\bf (i)]
for every $\beta>0$, $\Phi_\beta$ satisfies \eqref{eq:lin_wave} 
in $\mathcal{Q}$.
\item[\bf (ii)]
for every $\beta>0$, $\Phi_\beta$ satisfies $\pa_t\Phi_\beta=-\beta\Phi_{\beta+1}$ in $\mathcal{Q}$. 
\item[\bf (iii)](Upper bound)
the following inequality holds with the same constant $C_\nu$ as in Lemma \ref{lem:phi_lambda}:
\[
\Phi_\beta(x,t)\leq C_\nu U(x)t^{-\beta}
F\left(\frac{\beta}{2},\frac{\beta+1}{2},\frac{N}{2},\frac{|x|^2}{t^2}\right), 
\quad (x,t)\in \mathcal{Q}.
\]
\item[\bf (iv)](Lower bound) 
there exists a constant $C_\nu'>0$ such that for every 
$(x,t)\in \mathcal{Q}$ with $t\geq 1$,
\[
\Phi_\beta(x,t)\geq C_\nu 'U(x)t^{-\beta}.
\]
\item[\bf (v)](Large time behavior)
For every $x\in \Omega_0$, 
\begin{align*}
\lim_{t_0\to\infty}\Big(
t_0^{\beta}\Phi_\beta(x,t_0)
\Big)=U(x).
\end{align*}
\end{itemize}
\end{lemma}
\begin{proof}
{\bf (i)} 
Since for every $\lambda>0$, 
$e^{-\lambda t}\varphi_\lambda(x)$
satisfies \eqref{eq:lin_wave}, 
$\Phi_\beta$ is also the solution of the same problem. 
{\bf (ii)} By direct computation, we have
for every $(x,t)\in \mathcal{Q}$,
\begin{align*}
\pa_t\Phi_\beta(x,t)
&=
\frac{1}{\Gamma(\beta)}
\frac{\pa }{\pa t}
\left(
\int_0^{1} e^{-\lambda t}\varphi_\lambda(x)\lambda^{\beta-1}\,d\lambda
\right)
\\
&=
-
\frac{\beta}{\Gamma(\beta+1)}
\int_0^{1} e^{-\lambda t}\varphi_\lambda(x)\lambda^{(\beta+1)-1}\,d\lambda
\\
&=
-\beta\Phi_{\beta+1}(x,t).
\end{align*} 
{\bf (iii)}
Using Lemma \ref{lem:phi_lambda} (v), we deduce that 
for every $(x,t)\in \mathcal{Q}$,
\begin{align*}
\Phi_\beta(x,t)
&\leq \frac{C_\nu}{\Gamma(\beta)}
\int_0^{1} e^{-\lambda t}\Big(U(x)\psi_1(\lambda |x|)\Big)\lambda^{\beta-1}\,d\lambda
\\
&\leq \frac{C_\nu U(x)}{\Gamma(\beta)}
\int_0^{\infty} e^{-\lambda t}\psi_1(\lambda |x|)\lambda^{\beta-1}\,d\lambda
\\
&= C_\nu U(x)t^{-\beta} 
F\left(\frac{\beta}{2},\frac{\beta+1}{2},\frac{N}{2};\frac{|x|^2}{t^2}\right).
\end{align*}
{\bf (iv)}
Employing Lemma \ref{lem:phi_lambda} (iv)
with $\psi_1(z)\geq 1$ implies
that for every $(x,t)\in \mathcal{Q}$ with $t\geq 1$,
\begin{align*}
\Phi_\beta(x,t)
&\geq 
\frac{1}{\Gamma(\beta)}
\int_0^{1} e^{-\lambda t}U(x)\lambda^{\beta-1}\,d\lambda
\\
&= 
\frac{U(x)}{\Gamma(\beta)}
t^{-\beta}\int_0^{t} e^{-\mu}\mu^{\beta-1}\,d\mu
\\
&\geq 
\left(
\frac{1}{\Gamma(\beta)}
\int_0^{1} e^{-\mu}\mu^{\beta-1}\,d\mu
\right)
U(x)t^{-\beta}.
\end{align*}
{\bf (v)}
For $t_0>\max\{1,2|x|\}$, 
employing change of variables $\mu=\lambda t_0$ gives
\begin{align*}
t_0^{\beta}\Phi_\beta(x,t_0)
&=
\frac{t_0^{\beta}}{\Gamma(\beta)}
\int_0^1 e^{-\lambda t_0}\varphi_\lambda(x)\lambda^{\beta-1}\,d\lambda
\\
&=
\frac{1}{\Gamma(\beta)}
\int_0^{t_0} 
e^{-\mu}\varphi_{\mu t_0^{-1}}(x)\mu^{\beta-1}\,d\mu 
\\
&=
\frac{1}{\Gamma(\beta)}
\int_0^{\infty} 
\chi_{(0,t_0)}(\mu)e^{-\mu}\varphi_{\mu t_0^{-1}}(x)\mu^{\beta-1}\,d\mu,
\end{align*}
where $\chi_I$ denotes the indicator function on the interval $I$. 
Noting that $\varphi_{\mu t_0^{-1}}(x)\to U(x)$ as $t_0\to \infty$ 
(Lemma \ref{lem:phi_lambda} (iii))
and the consequence of Lemma \ref{lem:phi_lambda} (iv) as
\begin{align*}
\chi_{(0,t_0)}(\mu)e^{-\mu}\varphi_{\mu t_0^{-1}}(x)
&\leq 
e^{-\mu}U(x)\psi_1(\mu t_0^{-1}x)
\leq U(x)e^{-\mu+\frac{\mu|x|}{t_0}}
\leq U(x)e^{-\frac{\mu}{2}},
\end{align*}
from the dominated convergence theorem 
we obtain the desired convergence. 
\end{proof}
\section{Proof of blowup with lifespan estimates}\label{mainproof}

Here we give a proof of 
blowup phenomena with 
the sharp upper bound of lifespan estimates 
via the similar strategy as in \cite{IkSoWa_pre}. 
The difference of that is to use the 
special solutions satisfying Dirichlet boundary condition. 
For simplicity, 
we use $\Omega=\Omega_0$ in this section.
\begin{proof}[Proof of Theorem \ref{main}]
First we observe that for smooth function $\Psi$ on $\supp u$
satisfying Dirichlet boundary condition on $\pa\Omega$, 
we see by multiplying $\Psi$ to the equation in\ \eqref{NW}
and by using integration by parts that
\begin{align}\label{eq:TFM}
\int_{\Omega}|u|^p\Psi\,dx
=
\frac{d}{dt}
\int_{\Omega}(\pa_tu\Psi-u\pa_t\Psi)\,dx
+
\int_{\Omega}u(\pa_t^2\Psi-\Delta \Psi)\,dx.
\end{align}
We frequently use 
a cut-off function 
$\eta\in C^\infty([0,\infty);[0,1])$ satisfying 
$\eta(s)=1$ for $s\in [0,\frac{1}{2}]$ 
and
$\eta(s)=0$ for $s\in [1,\infty)$ 
with $\eta'(s)\leq 0$.
Also we assume without loss of generality that 
$T_\ep>1$ (otherwise the solution blows up until $t=1$).
\\
{\bf (Subcritical case $1<p<p_S(N)$)}\ 
Taking $\Psi=\eta_R(t)^{2p'}U(x)$ with $\eta_R(t)=\eta(t/R)$ in \eqref{eq:TFM}, we have
\begin{align*}
\int_\Omega
|u|^p  \eta_R^{2p'}
U\,dx
&=
\frac{d}{dt}
\int_\Omega
\Big(
\pa_tu \eta_R^{2p'}U
-
u U\pa_t[\eta_R^{2p'}]
\Big)\,dx
\\
&\quad
+
\int_\Omega
u\left(\pa_t^2 [\eta_R^{2p'}]U- \eta_R^{2p'}\Delta U\right)
\,dx
\\
&\leq
\frac{d}{dt}
\int_\Omega
\Big(
\pa_tu \eta_R^{2p'}U
-
u U\pa_t[\eta_R^{2p'}]
\Big)\,dx
\\
&\quad
+
CR^{-2}\int_\Omega
|u|[\eta_R^{2p'}]^{\frac{1}{p}}U
\,dx.
\end{align*}
Taking $R\in (1,T_\ep)$ and integrating it over $[0,T_\ep]$, we deduce
\begin{align*}
&\ep \int_\Omega gU\,dx
+
\int_0^{T_\ep}\!\!\int_\Omega
|u|^p  \eta_R^{2p'}
U\,dx\,dt
\\
&
\leq 
CR^{-2}
\int_0^{T_\ep}\!\!\int_\Omega
[|u|^p\eta_R^{2p'}]^{\frac{1}{p}}U
\,dx\,dt
\\&\leq 
\frac{C^{p'}}{p'}R^{-2p'}
\int_0^R\!\!\int_{\Omega(t)}
U
\,dx\,dt
+
\frac{1}{p}
\int_0^{T_\ep}\!\!\int_\Omega
|u|^p\eta_R^{2p'}U
\,dx\,dt, 
\end{align*}
where $\Omega(t)=\Omega\cap B(0,R_0+t)$. 
Noting that $U(x)\leq 1$ and $|\Omega(t)|\leq C(R_0+t)^{N}\leq C'R^{N}$, we have 
\begin{equation}\label{eq:volume}
p'\ep 
\int_\Omega gU\,dx
+
\int_0^{T_\ep}\!\!\int_\Omega
|u|^p\eta_R^{2p'}U
\,dx\,dt
\leq CR^{N-1-\frac{2}{p-1}}.
\end{equation}

Next we note by \eqref{eq:ass.pos} that 
\[
\int_\Omega g(x)\Big(t_0^{\beta}\Phi_{\beta}(x,t_0)\Big)\,dx
+
\beta \int_\Omega f(x)\Big(t_0^{\beta}\Phi_{\beta+1}(x,t_0)\Big)
\,dx\to 
\int_\Omega gU\,dx>0
\]
as $t_0\to \infty$. Therefore there exists $t_\beta> R_0$ such that 
\[
I_\beta=
\int_\Omega g(x)\widetilde{\Phi}_{\beta}(x,0)\,dx
+
\frac{\beta}{t_\beta}\int_\Omega f(x)\widetilde{\Phi}_{\beta+1}(x,0)
\,dx
\geq \frac{1}{2}\int gU\,dx
\]
with $\widetilde{\Phi}_{\beta}(x,t)=t_\beta^{\beta}\Phi_{\beta}(x,t_\beta+t)$.
Now we take $\Psi(x,t)=\eta_R^{2p'}\widetilde{\Phi}_{\beta}$
in \eqref{eq:TFM}. Then  
\begin{align*}
&\int_\Omega |u|^p\eta_R^{2p'}\widetilde{\Phi}_{\beta}(t)\,dx
\\
&=
\frac{d}{dt}
\int_\Omega 
\left(
\pa_tu \eta_R^{2p'}\widetilde{\Phi}_{\beta}(t)
+\frac{\beta}{t_\beta}u \eta_R^{2p'}\widetilde{\Phi}_{\beta+1}(t)
-2p'u \eta_R^{2p'-1}\eta_R'\widetilde{\Phi}_{\beta}(t)
\right)\,dx
\\
&\quad
+
2p'\int_\Omega u[\eta_R^*]^{\frac{2p'}{p}}\Big([(2p'-1)(\eta_R')^2+\eta_R''\eta_R]\widetilde{\Phi}_{\beta}(t)
+
\frac{2(N-1)}{t_\beta}\eta_R'\eta_R\widetilde{\Phi}_{\beta+1}(t)\Big)\,dx.
\end{align*}
Here we have introduced $\eta_R^*(t)=\eta^*(t/R)$ with 
$\eta^*(s)=\chi_{[\frac{1}{2},\infty)}\eta(s)$ ($\chi_I$ is the indicator function on the interval $I$). 
Integrating it over $[0,T_\ep]$, we see
\begin{align}
\label{eq:TFM_beta}
I_{\beta}
+
\int_0^{T_\ep}\!\!\int_\Omega |u|^p\eta_R^{2p'}\widetilde{\Phi}_{\beta}(t)\,dx\,dt
\leq 
C
\int_0^{T_\ep}\!\!\int_\Omega[|u|^p(\eta_R^*)^{2p'}]^{\frac{1}{p}}
\left(\frac{\widetilde{\Phi}_\beta}{R^2}+\frac{\widetilde{\Phi}_{\beta+1}}{t_\beta R}\right)\,dx\,dt.
\end{align}
Putting $\beta=N-1$ and using Lemma \ref{lem:Phi_beta} (iii)
with the formula $F(a,b,a;z)=F(b,a,a;z)=(1-z)^{-b}$, we have
\begin{align*}
\frac{\widetilde{\Phi}_{N-1}}{R^2}+\frac{\widetilde{\Phi}_{N}}{t_\beta R}
\leq 
CR^{-2-\beta}
\left(1-\frac{|x|^2}{t_\beta+t}\right)^{-\frac{N+1}{2}}U(x), \quad x\in {\rm supp}\,u(t), \ t\in (R/2,R).
\end{align*}
Computing the second integral on the right hand wide of the above inequality
with H\"older's inequality, 
we deduce
\begin{equation}\label{eq:concent}
\delta\left(
\frac{\ep}{2}
\int_\Omega gU\,dx
\right)^pR^{N-\frac{N-1}{2}p}\leq 
\int_0^T\!\!\int_\Omega|u|^p(\eta_R^*)^{2p'}U\,dx\,dt.
\end{equation}
Combining \eqref{eq:concent} with \eqref{eq:volume}, we obtain
\[
\delta\left(
\frac{\ep}{2}
\int_\Omega gU\,dx
\right)^p\leq CR^{-\frac{\gamma(N,p)}{2(p-1)}}
\]
which implies the desired upper bound for lifespan of $u$ when $1<p<p_S(N)$. 
\\
{\bf (Critical case $p=p_S(N)$)}\ In view of \eqref{eq:concent} together with Lemma \ref{lem:Phi_beta} (iv), we have 
\begin{equation}\label{eq:concent_crit}
\tilde{\delta}\left(
\frac{\ep}{2}
\int_\Omega gU\,dx
\right)^p\leq 
\int_0^{T_\ep}\!\!\int_\Omega|u|^p(\eta_R^*)^{2p'}\widetilde{\Phi}_{\beta_p}\,dx\,dt,
\end{equation}
where $\beta_p=\frac{N-1}{2}-\frac{1}{p}=N-\frac{N-1}{2}p>0$ by the condition $p=p_S(N)$. 
Take $\beta=\beta_p$ in \eqref{eq:TFM_beta}. 
Then 
\begin{align*}
&\frac{\ep}{2}\int_\Omega gU\,dx
+
\int_0^{T_\ep}\!\!\int_\Omega |u|^p\eta_R^{2p'}\widetilde{\Phi}_{\beta}\,dx\,dt
\\
&\leq 
C
\left(
\int_{\frac{R}{2}}^R\!\!\int_{\Omega(t)}
\left(\frac{1}{R^2}+\frac{\widetilde{\Phi}_{\beta+1}}{t_\beta R\widetilde{\Phi}_\beta}\right)^{p'}\widetilde{\Phi}_\beta\,dx\,dt
\right)^{\frac{1}{p'}}
\left(
\int_0^{T_\ep}\!\!\int_\Omega |u|^p(\eta_R^*)^{2p'}\widetilde{\Phi}_\beta\,dx\,dt
\right)^{\frac{1}{p}}.
\end{align*}
Applying Lemma \ref{lem:Phi_beta} (iii) and (iv), we obtain
\begin{align}
\label{eq:volume_crit}
&
\int_0^{T_\ep}\!\!\int_\Omega |u|^p\eta_R^{2p'}\widetilde{\Phi}_{\beta}\,dx\,dt
\leq 
C(\log R)^{\frac{1}{p'}}
\left(
\int_0^{T_\ep}\!\!\int_\Omega |u|^p(\eta_R^*)^{2p'}\widetilde{\Phi}_\beta\,dx\,dt
\right)^{\frac{1}{p}}.
\end{align}
By introducing the function 
\[
Y(R)=
\int_0^R
\left(\int_0^{T_\ep}\!\!\int_\Omega |u|^p(\eta_\rho^*)^{2p'}\widetilde{\Phi}_\beta\,dx\,dt\right)\rho^{-1}\,d\rho
\]
(as in \cite[Lemma 3.9]{IkSoWa_pre}), the inequalities \eqref{eq:concent_crit} and \eqref{eq:volume_crit} can be translated into 
\begin{align*}
\begin{cases}
\tilde{\delta}\left(
\dfrac{\ep}{2}
\displaystyle\int_\Omega gU\,dx
\right)^p\leq RY'(R), 
\\
Y(R)^p\leq (\log R)^{p-1}RY'(R)
\end{cases}
\end{align*}
for every $R\in (1,T_\ep)$. Employing \cite[Lemma 2.10]{IkSoWa_pre}, we obtain
\[
T\leq \exp\left(C\left(
\dfrac{\ep}{2}
\displaystyle\int_\Omega gU\,dx
\right)^{-p(p-1)}\right). 
\]
This gives the desired upper bound of the lifespan of $u$ in the critical case $p=p_S(N)$.
The proof is complete.
\end{proof}



\end{document}